\documentclass{article}[12pt]

\usepackage{amsmath,amssymb,amscd,amsthm}
\usepackage{graphics}
\usepackage[dvips]{graphicx}
\usepackage{latexsym}
\usepackage{amsmath,amsfonts}
\usepackage{dsfont}
\usepackage{enumerate}
\usepackage{cite}
\usepackage{tikz}
\usepackage{hyperref}

\theoremstyle{plain}
\newtheorem{theorem}{Theorem}[section]
\newtheorem{corollary}[theorem]{Corollary}
\newtheorem{proposition}[theorem]{Proposition}
\newtheorem{lemma}[theorem]{Lemma}

\theoremstyle{definition}

\newtheorem{remark}[theorem]{Remark}

\DeclareMathOperator{\Char}{char}

\DeclareMathOperator{\Id}{id}

\newcommand{\FI}{FI(X,K)}

\begin{document}

\title{\textbf{Involutions of the second kind on finitary incidence algebras}}
\author{\'{E}rica Z. Fornaroli\thanks{ezancanella@uem.br} \\
\\
\textit{{\normalsize Departamento de Matem\'atica, Universidade
Estadual de Maring\'a,}} \\ \textit{{\normalsize Avenida Colombo,
5790, Maring\'a, PR, 87020-900, Brazil.}}}
\date{}

\maketitle

\begin{abstract}
Let $K$ be a field and $X$ a connected partially ordered set. In the first part of this paper, we show that the finitary incidence algebra $FI(X,K)$ of $X$ over $K$ has an involution of the second kind if and only if $X$ has an involution and $K$ has an automorphism of order $2$. We also give a characterization of the involutions of the second kind on $FI(X,K)$. In the second part, we give necessary and sufficient conditions for two involutions of the second kind on $FI(X,K)$ to be equivalent in the case where $\Char K\neq 2$ and every multiplicative automorphism of $\FI$ is inner.


\end{abstract}

\maketitle

\section{Introduction}

Let $X$ be a partially ordered set and let $K$ be a field. Involutions of the first kind on the finitary incidence algebra $FI(X,K)$ were studied in \cite{BFS11,BFS12,BFS14,BL,DKL,Scharlau,Spiegel05,Spiegel08}. In \cite{BFS12} was proved that $FI(X,K)$ has a $K$-linear involution if and only if $X$ has an involution. The description of such involutions was given in \cite{BFS14}. When $\Char K\neq 2$ and $X$ has an all-comparable element, the classification of the $K$-linear involutions on $FI(X,K)$ was obtained firstly for the case where $X$ is finite \cite{BFS11}, and later, locally finite~\cite{BFS12}. After, in~\cite{BFS14}, the classification was generalized for the case where $X$ is connected (not necessarily locally finite) and every multiplicative automorphism of $FI(X,K)$ is inner. If $X$ is connected, then $FI(X,K)$ is a central algebra, therefore the classification of all involutions of the first kind was given in \cite{BFS14}.

If $X$ is a finite chain of cardinality $n\geq 1$, then $FI(X,K)$ is isomorphic to the algebra $UT_n(K)$ of the $n\times n$ upper triangular matrices over $K$. When $\Char K\neq 2$, the involutions of the first kind on $UT_n(K)$ were classified \cite{DKL} and those of the second kind, in \cite{US}.

By \cite[Lemma~2.1]{US}, if two involutions on an algebra are equivalent, they are of the same kind. Thus, the classification of involutions of the second kind on $FI(X,K)$ in the case where $\Char K\neq 2$, $X$ is connected and every multiplicative automorphism of $FI(X,K)$ is inner completes the classification of involutions on $FI(X,K)$ under these assumptions.

In Section~\ref{Sec-Preliminaries} we fix the notations, recall some results about involutions and finitary incidence algebras, and we also prove some general results. If $X$ is connected, we prove in Section~\ref{Sec-Involutions} that $FI(X,K)$ has an involution of the second kind if and only if $X$ has an involution and $K$ has an involution of the second kind (Theorem~\ref{has_involution_iff}). We also give a description of the involutions of the second kind on $FI(X,K)$ (Theorem~\ref{DecoInv}). Finally, in Section~\ref{Classification}, we consider a field $K$ of characteristic different from $2$ and a connected partially ordered set $X$ such that every multiplicative automorphism of $FI(X,K)$ is inner and we give necessary and sufficient conditions for two involutions on $FI(X,K)$ to be equivalent via inner automorphisms (Theorems~\ref{necessary_inner}, \ref{x_3vazio}, \ref{x_3naovazio} and \ref{x_3unitario}). Then we use such a classification to obtain the general classification (Theorem~\ref{teogeralclassiinvolu}, Corollary~\ref{cor_general}, Theorems~\ref{x_3lessthanorequalto1} and \ref{x_3morethan1}). The results obtained in this section generalize the classification of involutions of the second kind on $UT_n(K)$ given in \cite{US} (that is, in the case where $X$ is a finite chain).

\section{Preliminaries }\label{Sec-Preliminaries}

\subsection{Involutions on a arbitrary algebra}

Let $K$ be a field and $K^\times=K\backslash\{0\}$. Throughout the paper $K$-algebras are associative with unity. The center of a $K$-algebra $A$ is denoted by $Z(A)$ and the set of invertible elements of $A$ is denoted by $U(A)$. For each $a \in U(A)$, $\Psi_a$ denotes the inner automorphism defined by $a$, i.e., $\Psi_a:A\to A$ is such that $\Psi_a(x) = axa^{-1}$ for all $x \in A$.

An \emph{involution} on a $K$-algebra $A$ is an additive map $\rho: A\to A$ such that $\rho(ab)=\rho(b)\rho(a)$ and $\rho^2(a)=a$ for all $a,b\in A$. Note that $\rho(a)\in Z(A)$ for all $a\in Z(A)$. If $\rho$ is the identity map on $Z(A)$, then it is an involution of the \emph{first kind}, otherwise, it is of the \emph{second kind}. It is easy to see that $\rho(1)=1$ and $\rho(a^{-1})=\rho(a)^{-1}$ for all $a\in U(A)$. An element $a\in A$ is \emph{$\rho$-symmetric} if $\rho(a)=a$, and \emph{$\rho$-unitary} if $a\rho(a)=\rho(a)a=1$.

Two involutions $\rho_1$ and $\rho_2$ on a $K$-algebra $A$ are \emph{equivalent} if there is an automorphism $\Phi$ of $A$ such that $\Phi\circ \rho_1 = \rho_2\circ \Phi$. In this case, $\rho_1$ and $\rho_2$ are of the same kind, by \cite[Lemma~2.1]{US}. Moreover, if $A$ is central, then $\rho_1|_K=\rho_2|_K$ by the same lemma.

We present below some results that holds for involutions of the first or second kind on an arbitrary $K$-algebra $A$.

\begin{proposition}\label{idempotent}
Let $\rho$ be an involution on a $K$-algebra $A$. If $e\in A$ is an idempotent, then $\rho(e)$ is an idempotent. Moreover, if the idempotent $e$ is primitive, then so is $\rho(e)$.
\end{proposition}
\begin{proof}
It is clear that if $e\in A$ is an idempotent, then so is $\rho(e)$. If $e$ is a primitive idempotent, let $s\in A$ be an idempotent such that $s\rho(e)=\rho(e)s=s$. Then $e\rho(s)=\rho(s)e=\rho(s)$ and $\rho(s)$ is an idempotent. Since $e$ is primitive, it follows that $\rho(s)=0$ or $\rho(s)=e$, therefore $s=0$ or $s=\rho(e)$. Thus, $\rho(e)$ is a primitive idempotent.
\end{proof}

\begin{lemma}\label{rho_psi_outra_psi_rho}
Let $\rho$ be an involution on a $K$-algebra $A$ and $u\in U(A)$. Then $\rho\circ\Psi_u=\Psi_{\rho(u^{-1})}\circ\rho$.
\end{lemma}
\begin{proof}
It is the same as for \cite[Lemma~4]{BFS11}.
\end{proof}

\begin{lemma}\label{inner_iff}
Let $\rho_1$ and $\rho_2$ be involutions on a $K$-algebra $A$. There is an inner automorphism $\Psi$ of $A$ such that $\Psi\circ\rho_1=\rho_2\circ\Psi$ if and only if $\rho_1=\Psi_{v\rho_2(v)}\circ\rho_2$ for some $v\in U(A)$.
\end{lemma}
\begin{proof}
It is the same as for \cite[Lemma~5]{BFS11} just by replacing \cite[Lemma~4]{BFS11} by Lemma~\ref{rho_psi_outra_psi_rho}.
\end{proof}

\begin{proposition} \label{propAutInner}
Let $\rho$ be an involution on a $K$-algebra $A$ and $u\in U(A)$.
\begin{enumerate}
\item[(i)] $\Psi_u \circ \rho$ is an involution on $A$ if and only if $\rho(u)= cu$ for some $c\in Z(A)$.
\item[(ii)] If $\rho(u)=ku$ for some $k\in K$, then $\Psi_u \circ \rho$ is an involution on $A$ and $k$ is $\rho$-unitary. On the other hand, if $A$ is a central algebra and $\Psi_u \circ \rho$ is an involution, then $\rho(u)=ku$ for some $\rho$-unitary $k\in K$.
\end{enumerate}
\end{proposition}
\begin{proof}
(i) Clearly, $\Psi_u \circ \rho$ is additive and satisfies $(\Psi_u \circ \rho)(ab)=(\Psi_u \circ \rho)(b)(\Psi_u \circ \rho)(a)$. Moreover, for any $x\in A$,
$$(\Psi_u \circ \rho)^2(x) =(\Psi_u \circ \rho)(u\rho(x)u^{-1})=\Psi_u(\rho(u)^{-1}x\rho(u))= u\rho(u)^{-1}x\rho(u)u^{-1}.$$
Thus, $(\Psi_u \circ \rho)^2(x)=x$ if and only if $\rho(u)u^{-1}x=x\rho(u)u^{-1}$. Therefore, $\Psi_u \circ \rho$ is an involution if and only if $\rho(u)u^{-1}\in Z(A)$.

(ii) If $\rho(u)=ku$ for some $k\in K$, then $\Psi_u \circ \rho$ is an involution, by (i), and
$$u=\rho^2(u)=\rho(ku)=\rho(uk)=\rho(k)\rho(u)=\rho(k)ku.$$
Thus $\rho(k)k=1$.

If $A$ is a central algebra and $\Psi_u \circ \rho$ is an involution, then there is $k\in K$ such that $\rho(u)=ku$, by (i), and $\rho(k)k=1$ as proved in the previous paragraph.
\end{proof}

\subsection{Involutions of the second kind on a field of characteristic different from $2$}

Now, let $K$ be a field of characteristic different from $2$ and consider $K$ as a $K$-algebra. Suppose there is an involution $\ast: K\to K$ of the second kind. We will denote $\ast(a)$ by $a^{\ast}$ for each $a\in K$. Let
$$K_0=\{a\in K : a^{\ast}=a\}.$$
Then $K_0$ is a proper subfield of $K$ such that $[K : K_0]=2$ and there is a linear basis $\{1,i\}$ of $K$ over $K_0$ such that $i^2\in K_0$ and $i^{\ast}=-i$, by \cite[Lemma~2.5]{US}, since $\ast$ is an automorphism of order $2$.

\begin{proposition}\label{k=a_ast_a-1}
Let $K$ be a field of characteristic different from $2$ with an involution $\ast$ of the second kind. For each $\ast$-unitary $k\in K^{\times}$ there is $a\in K^{\times}$ such that $k=a^{\ast}a^{-1}$.
\end{proposition}
\begin{proof}
Let $k\in K^{\times}$ be a $\ast$-unitary element. Let $\{1,i\}$ be a linear basis of $K$ over $K_0$ such that $i^2\in K_0$ and $i^{\ast}=-i$. Then there are $a,b\in K_0$ such that $k=a+bi$.

If $b=0$, then $k^{\ast}=a=k$, that is, $k^{-1}=k$. Thus $k=\pm 1$ and therefore
$$k=1=1^{\ast}1^{-1} \text{ or } k=-1=i^{\ast}i^{-1}.$$

Suppose $b\neq 0$. We have $1=kk^{\ast}=(a+bi)(a-bi)=a^2-b^2i^2$ which implies
\begin{equation}\label{bqaq}
  b^2i^2=a^2-1.
\end{equation}
Consider $l=(1+a)-bi\in K^{\times}$. Applying \eqref{bqaq} we obtain
\begin{align*}
  l^{\ast}l^{-1} & = [(1+a)+bi]\cdot\frac{(1+a)+bi}{(1+a)^2-b^2i^2}=\frac{[(1+a)+bi]^2}{(1+a)^2-b^2i^2} \\
                 & = \frac{(1+a)^2+2(1+a)bi+b^2i^2}{1+2a+a^2-b^2i^2}  \\
                 & = \frac{1+2a+a^2+2bi+2abi+a^2-1}{1+2a+a^2-a^2+1} \\
                 & = \frac{2a(1+a)+2bi(1+a)}{2(1+a)}=a+bi=k.
\end{align*}
\end{proof}

We will also consider the set
$$K_1=\{aa^{\ast} : a\in K^{\times}\}.$$
It is easy to see that $K_1$ is a subgroup of the multiplicative group $K_0^{\times}=K_0\backslash \{0\}$.

\subsection{Posets and finitary incidence algebras}

Let $(X, \leq)$ be a partially ordered set (poset, for short). Given $x,y \in X$, the \emph{interval} from $x$ to $y$ is the set $[x,y]= \{z \in X :$ $x \leq z \leq y\}$. If all intervals of $X$ are finite, then $X$ is said to be \emph{locally finite}. The poset $X$ is \emph{connected} if for any $x,y \in X$ there exist a positive integer $n$ and $x=x_0,x_1,\ldots, x_n=y$ in $X$ such that $x_i\leq x_{i+1}$ or $x_{i+1}\leq x_i$ for $i=0,1,\ldots,n-1$.

A map $\lambda:X \to X$ is an \emph{involution} on $X$ if $\lambda^2(x)=x$ for all $x\in X$, and for any $x,y \in X$, the following property is satisfied:
$$ x\leq y \Leftrightarrow \lambda(y) \leq \lambda(x).$$

If the poset $X$ has an involution $\lambda$, then, by \cite[Theorem~4.7]{BFS14}, there is a triple of disjoint subsets $(X_1, X_2, X_3)$ of $X$ with $X = X_1 \cup X_2 \cup X_3$ satisfying:
\begin{enumerate}
\item[(i)] $X_3=\{ x \in X : \lambda(x) = x\}$;
\item[(ii)] if $x \in X_1$ ($X_2$), then $\lambda(x) \in X_2$ ($X_1$);
\item[(iii)] if $x \in X_1$ ($X_2$) and $y \leq x$ ($x \leq y$), then $y \in X_1$ ($X_2$).
\end{enumerate}
In this case, $(X_1, X_2, X_3)$ is called a $\lambda$-\emph{decomposition of} $X$.

Let $X$ be a poset and $K$ a field. The \emph{incidence space} of $X$ over $K$ is the $K$-space of functions $I(X,K)=\{f:X\times X\to K : f(x,y)=0 \text{ if } x\nleq y\}$. Let $FI(X,K)$ be the subspace of $I(X,K)$ formed by the functions $f$ such that for any $x\leq y$ in $X$, there is only a finite number of subintervals $[u,v]\subseteq[x,y]$ such that $u\neq v$ and $f(u,v)\neq 0$. Then $FI(X,K)$ is a $K$-algebra with the (convolution) product
	$$(fg)(x,y)=\sum_{x\leq z\leq y}f(x,z)g(z,y),$$
for any $f, g\in \FI$, called the \emph{finitary incidence algebra} of $X$ over $K$. Furthermore, if $f \in FI(X,K)$ and $g\in I(X,K)$, then $fg, gf \in I(X,K)$, and with this action of $FI(X,K)$ on $I(X,K)$ on the left and on the right, $I(X,K)$ is an $FI(X,K)$-bimodule~\cite{KN}.

The unity of $FI(X,K)$ is the function $\delta$ defined by $\delta(x,y)=1$ if $x=y$, and $\delta(x,y)=0$ otherwise. An element $f \in FI(X,K)$ is invertible if and only if $f(x,x) \neq 0$ for all $x \in X$, by \cite[Theorem~2]{KN}. Given $x,y\in X$ with $x\leq y$, we denote by $e_{xy}$ the element of $FI(X,K)$ defined by $e_{xy}(x,y)=1$ and $e_{xy}(u,v)=0$ if $(u,v)\neq (x,y)$, and we write $e_x$ for $e_{xx}$. Then for any $f\in FI(X,K)$,
\begin{equation}\label{e_xfe_y}
e_xfe_y=\begin{cases}
f(x,y)e_{xy} & \text{if } x\leq y\\
0 & \text{otherwise}
\end{cases}.
\end{equation}

\begin{proposition}\label{propcenFI}
Let $X$ be a poset and $K$ a field. The finitary incidence algebra $FI(X,K)$ is central if and only if $X$ is connected.
\end{proposition}
\begin{proof}
The proof is the same as for the case where $X$ is locally finite. (See Theorem~1.3.13 and Corollary~1.3.14 of \cite{SO97}).
\end{proof}

We recall that if $\alpha$ is an automorphism $X$, then $\alpha$ induces an automorphism $\widehat{\alpha}$ of $FI(X,K)$ by $\widehat{\alpha}(f)(x,y) = f(\alpha^{-1}(x),\alpha^{-1}(y))$, for all $f\in FI(X,K)$ and $x,y\in X$. An element $\sigma \in I(X,K)$ such that $\sigma(x,y)\neq 0$ for all $x \leq y$, and $\sigma(x,y)\sigma(y,z)=\sigma(x,z)$ whenever $x\leq y\leq z$, determines an automorphism $M_\sigma$ of $FI(X,K)$ by $M_{\sigma}(f)(x,y)=\sigma(x,y)f(x,y)$, for all $f \in FI(X,K)$ and $x,y \in X$ (see \cite[Proposition~2.2]{BFS14}). Such an automorphism $M_{\sigma}$ is called \emph{multiplicative}. The next theorem follows from \cite[Lemma~3]{AutK}.

\begin{theorem} \label{DecoAutFIP}
 Let $X$ be a poset and let $K$ be a field. If $\Phi$ is an automorphism of $FI(X,K)$, then $\Phi = \Psi \circ M \circ \widehat{\alpha}$, where $\Psi$ is an inner automorphism, $M$ is a multiplicative automorphism, and $\widehat{\alpha}$ is the automorphism of $FI(X,K)$ induced by an automorphism $\alpha$ of $X$.
\end{theorem}

\section{Involutions on $FI(X,K)$} \label{Sec-Involutions}

From now on, $K$ is a field and $X$ is a poset. In this section, we show that in the case where $X$ is connected, $FI(X,K)$ has an involution of the second kind if and only if $X$ has an involution and $K$ has an involution of the second kind (Theorem~\ref{has_involution_iff}). In this case, we also give a characterization of the involutions of the second kind on $FI(X,K)$ (Theorem~\ref{DecoInv}).

Note that if $\ast: K\to K$ is an involution, then $\ast=\Id_K$ or $\ast$ is an automorphism of order $2$, depending on whether $\ast$ is of the first or second kind. We will denote $\ast(a)$ by $a^{\ast}$ for each $a\in K$.

\begin{lemma}\label{rolbstar_involution}
If $\ast$ is an involution on $K$ and $\lambda$ is an involution on $X$, then $\rho_{\lambda}^{\ast}:FI(X,K)\to FI(X,K)$ given by
$$ \rho_{\lambda}^{\ast}(f)(x,y)=[f(\lambda(y),\lambda(x))]^{\ast},$$
for all $f\in FI(X,K)$ and $x,y\in X$, is an involution.
\end{lemma}
\begin{proof}
 Firstly note that $\rho_{\lambda}^{\ast}(f)\in FI(X,K)$ for all $f\in FI(X,K)$; indeed, if $x\nleq y$, then $\lambda(y)\nleq \lambda(x)$ and so $\rho_{\lambda}^{\ast}(f)(x,y)=0$. Moreover, if $[u,v]\subseteq [x,y]$ is such that $u\neq v$ and $\rho_{\lambda}^{\ast}(f)(u,v)\neq 0$, then $[\lambda(v),\lambda(u)]\subseteq [\lambda(y),\lambda(x)]$, $\lambda(v)\neq \lambda(u)$ and $f(\lambda(v),\lambda(u))\neq 0$. Thus, since $f\in FI(X,K)$, then $\rho_{\lambda}^{\ast}(f)\in FI(X,K)$.

 Since $\ast$ is additive, then so is $\rho_{\lambda}^{\ast}$. Let $f,g\in FI(X,K)$ and $x\leq y$ in $X$. We have
 \begin{align*}
   \rho_{\lambda}^{\ast}(fg)(x,y) & = [(fg)(\lambda(y),\lambda(x))]^{\ast} \\
                                  & = \left[\sum_{\lambda(y) \leq \lambda(z)\leq \lambda(x)} f(\lambda(y),\lambda(z))g(\lambda(z),\lambda(x))\right]^{\ast}\\
                                  & = \sum_{x\leq z \leq y} [g(\lambda(z),\lambda(x))]^{\ast}[f(\lambda(y),\lambda(z))]^{\ast} \\
                                  & = \sum_{x\leq z \leq y} \rho_{\lambda}^{\ast}(g)(x,z)\rho_{\lambda}^{\ast}(f)(z,y) \\
                                  & = (\rho_{\lambda}^{\ast}(g)\rho_{\lambda}^{\ast}(f))(x,y)
 \end{align*}
and
$$\rho_{\lambda}^{\ast}(\rho_{\lambda}^{\ast}(f))(x,y)=[\rho_{\lambda}^{\ast}(f)(\lambda(y),\lambda(x))]^{\ast}
             =\{[f(\lambda(\lambda(x)),\lambda(\lambda(y)))]^{\ast}\}^{\ast}=f(x,y).$$
Thus, $\rho_{\lambda}^{\ast}(fg)=\rho_{\lambda}^{\ast}(g)\rho_{\lambda}^{\ast}(f)$ and $\rho_{\lambda}^{\ast}(\rho_{\lambda}^{\ast}(f))=f$ and so $\rho_{\lambda}^{\ast}$ is an involution.
\end{proof}

\begin{remark}
Note that if $\ast$ is an involution on $K$ and $\lambda$ is an involution on $X$, then
\begin{equation}\label{rolb_a}
  \rho_{\lambda}^{\ast}(a)=a^{\ast}, \forall a\in K
\end{equation}
and
\begin{equation}\label{rolb_e_x}
  \rho_{\lambda}^{\ast}(e_{x})= e_{\lambda(x)}, \forall x\in X.
\end{equation}
\end{remark}

\begin{lemma}\label{rolbstar_involution_scnd}
Let $X$ be a connected poset with an involution $\lambda$ and let $\ast$ be an involution on $K$. Then $\rho_{\lambda}^{\ast}$ is of the second kind if and only if $\ast$ is of the second kind.
\end{lemma}
\begin{proof}
  Since $X$ is connected, then $FI(X,K)$ is central by Proposition~\ref{propcenFI}. Thus, by \eqref{rolb_a}, $\rho_{\lambda}^{\ast}$ is of the second kind if and only if $\ast$ is of the second kind.
\end{proof}

\begin{theorem}\label{has_involution_iff}
  Let $K$ be a field and $X$ a connected poset. Then $FI(X,K)$ has an involution of the second kind if and only if $K$ has an involution of the second kind and $X$ has an involution.
\end{theorem}
\begin{proof}
If $K$ has an involution $\ast$ of the second kind and $X$ has an involution $\lambda$, then $\rho_{\lambda}^{\ast}$ is an involution of the second kind on $FI(X,K)$, by Lemmas~\ref{rolbstar_involution} and \ref{rolbstar_involution_scnd}.

Conversely, suppose $\rho$ is an involution of the second kind on $FI(X,K)$. Given $x\in X$, by \cite[Lemma~1]{KN}, $e_x$ is a primitive idempotent of $FI(X,K)$, and therefore $\rho(e_x)$ is also a primitive idempotent, by Proposition~\ref{idempotent}. Again by \cite[Lemma~1]{KN}, there is a unique $y\in X$ such that $\rho(e_x)$ is conjugate to $e_y$. Thus, $\rho$ induces a map $\lambda: X\to X$ such that, given $x\in X$, $\rho(e_x)$ is conjugate to $e_{\lambda(x)}$. We will show that $\lambda$ is an involution.

For each $x\in X$, let $f_x\in U(FI(X,K))$ such that $\rho(e_x)=f_xe_{\lambda(x)}f_x^{-1}$. We have
  \begin{align*}
  e_x=\rho^2(e_x) & =\rho(f_xe_{\lambda(x)}f_x^{-1})=\rho(f_x)^{-1}\rho(e_{\lambda(x)})\rho(f_x)\\
                  & = \rho(f_x)^{-1}f_{\lambda(x)}e_{\lambda^2(x)}f_{\lambda(x)}^{-1}\rho(f_x)\\
                  & = [\rho(f_x)^{-1}f_{\lambda(x)}]e_{\lambda^2(x)}[\rho(f_x)^{-1}f_{\lambda(x)}]^{-1}
  \end{align*}
and hence $\lambda^2(x)=x$. Let $x,y\in X$. By \eqref{e_xfe_y},
 \begin{align*}
            x\leq y   & \Leftrightarrow  e_xFI(X,K)e_y\neq 0\\
                      & \Leftrightarrow \rho(e_y)FI(X,K)\rho(e_x)\neq 0\\
                      & \Leftrightarrow f_ye_{\lambda(y)}f_y^{-1}FI(X,K)f_xe_{\lambda(x)}f_x^{-1}\neq 0\\
                      & \Leftrightarrow f_ye_{\lambda(y)}FI(X,K)e_{\lambda(x)}f_x^{-1}\neq 0\\
                      & \Leftrightarrow e_{\lambda(y)}FI(X,K)e_{\lambda(x)}\neq 0\\
                      & \Leftrightarrow \lambda(y)\leq {\lambda(x)}.
  \end{align*}
It follows that $\lambda$ is an involution.

Finally, since $FI(X,K)$ is a central algebra, by Proposition~\ref{propcenFI}, then $\rho|_K$ is an involution on $K$, which is of the second kind because $\rho$ is of the second kind.
\end{proof}

To give a characterization of the involutions of the second kind on $FI(X,K)$ we need some lemmas.

\begin{lemma}\label{analogo3.3}
Let $K$ be a field and $X$ a connected poset. Let $\rho$ be an involution of the second kind on $FI(X,K)$ that induces the involution $\lambda$ on $X$. Let $x,y\in X$ with $x\leq y$. Then we have:
\begin{description}
  \item (i) $\rho(e_y)(\lambda(y),\lambda(x))=0 \Leftrightarrow \rho(e_{\lambda(x)})(x,y)=0$;
  \item (i) $\rho(e_x)(\lambda(y),\lambda(x))=0 \Leftrightarrow \rho(e_{\lambda(y)})(x,y)=0$.
\end{description}
\end{lemma}
\begin{proof}
  It is the same as for \cite[Lemma~3.3]{BFS14}.
\end{proof}

Let $Y\subseteq X$ be an arbitrary subset and consider $e_Y\in FI(X,K)$ defined by $e_Y(u,v)=1$ if $u=v\in Y$ and $e_Y(u,v)=0$ otherwise. Note that if $Y=\{x\}$, then $e_{\{x\}}=e_x$.

\begin{lemma}\label{analogo3.4}
Let $K$ be a field and $X$ a connected poset. Let $\rho$ be an involution of the second kind on $FI(X,K)$ that induces the involution $\lambda$ on $X$. Let $x,y\in X$ with $x\leq y$. Then, for any $Y\subseteq X$, $\rho(e_Y)(\lambda(y),\lambda(x))=\rho(e_{Y'})(\lambda(y),\lambda(x))$, where $Y'=\{z\in Y : \rho(e_z)(\lambda(z),\lambda(x))\neq 0 \text{ and } \rho(e_z)(\lambda(y),\lambda(z))\neq 0\}$.
\end{lemma}
\begin{proof}
  The proof is the same as for \cite[Lemma~3.4]{BFS14}.
\end{proof}

\begin{theorem}\label{DecoInv}
Let $K$ be a field and $X$ a connected poset. If $\rho$ is an involution of the second kind on $FI(X,K)$, then $\rho = \Psi \circ M \circ \rho_{\lambda}^{\ast}$, where $\Psi$ is an inner automorphism and $M$ is a multiplicative automorphism of $FI(X,K)$, $\lambda$ is the involution on $X$ induced by $\rho$, and $\ast=\rho|_K$.
\end{theorem}
\begin{proof}
Let $\lambda$ be the involution on $X$ induced by $\rho$ and $\ast=\rho|_K$. Consider the involution $\rho_{\lambda}^{\ast}$ on $FI(X,K)$. Define $f\in I(X,K)$ by \begin{equation}\label{f}
f(u,v)=(\rho\circ \rho_{\lambda}^{\ast})(e_u)(u,v)=\rho(e_{\lambda(u)})(u,v),
\end{equation}
where the last equality follows from \eqref{rolb_e_x}. We can show that $f\in FI(X,K)$ as was done in the proof of Theorem~3.5 of \cite{BFS14}, replacing Lemmas~3.3 and 3.4 of \cite{BFS14} by Lemmas~\ref{analogo3.3} and \ref{analogo3.4}, respectively.

Let $y\in X$. For any $u\leq v$ in $X$ we have
\begin{align*}
[f(\rho\circ\rho_{\lambda}^{\ast})(e_y)](u,v) & = \sum_{u\leq t\leq v}f(u,t)(\rho\circ\rho_{\lambda}^{\ast})(e_y)(t,v)\\
                                          & = \sum_{u\leq t\leq v}(\rho\circ\rho_{\lambda}^{\ast})(e_u)(u,t)
                                          (\rho\circ\rho_{\lambda}^{\ast})(e_y)(t,v)\\
                                          & = [(\rho\circ\rho_{\lambda}^{\ast})(e_u)\cdot
                                          (\rho\circ\rho_{\lambda}^{\ast})(e_y)](u,v)\\
                                          & = [(\rho\circ\rho_{\lambda}^{\ast})(e_ue_y)](u,v)\\
                                          & = \left\{\!\!\begin{array}{ll}
                                                0 & \text{if } u\neq y\\
                                                (\rho\circ\rho_{\lambda}^{\ast})(e_y)(y,v) & \text{if } u=y\\
                                                \end{array}\right.\\
                                          & = \left\{\!\!\begin{array}{ll}
                                                0 & \text{if } u\neq y\\
                                                f(y,v) & \text{if } u=y\\
                                                \end{array}\right.\\
                                          & = \sum_{u\leq t\leq v}e_y(u,t)f(t,v)\\
                                          & = (e_yf)(u,v).
\end{align*}
Therefore $f(\rho\circ\rho_{\lambda}^{\ast})(e_y)=e_yf$, for all $y\in X$.

For each $x\in X$, let $g_x\in U(FI(X,K))$ such that $\rho(e_x)=g_xe_{\lambda(x)}g_x^{-1}$. Then $f(x,x)=\rho(e_{\lambda(x)})(x,x)=(g_{\lambda(x)}e_xg_{\lambda(x)}^{-1})(x,x)=1$, and so $f\in U(FI(X,K))$. Thus, for all $y\in X$, $(\rho\circ\rho_{\lambda}^{\ast})(e_y)=f^{-1}e_yf=\Psi_{f^{-1}}(e_y)$ and therefore
\begin{equation}\label{taquasemultiplicativo}
(\Psi_f\circ\rho\circ \rho_{\lambda}^{\ast})(e_y)=e_y.
\end{equation}

Now, $\rho|_K=\ast=\rho_{\lambda}^{\ast}|_K$, by \eqref{rolb_a}. Thus, $\rho\circ \rho_{\lambda}^{\ast}$ is an automorphism of $FI(X,K)$, by \cite[Lemma~2.2]{US}. Therefore, it follows from \eqref{taquasemultiplicativo} and \cite[Proposition~2.3]{BFS14}, that $\Psi_f\circ\rho\circ \rho_{\lambda}^{\ast}$ is a multiplicative automorphism, that is, there exists a multiplicative $\sigma\in I(X,K)$ such that $\Psi_f\circ\rho\circ \rho_{\lambda}^{\ast}=M_{\sigma}$ and then $\rho=\Psi_{f^{-1}}\circ M_{\sigma} \circ \rho_{\lambda}^{\ast}$.
\end{proof}

\section{Classification of involutions}\label{Classification}

From now on, the field $K$ has characteristic different from $2$ and $X$ is a connected poset such that every multiplicative automorphism of $FI(X,K)$ is inner. Thus, by Proposition~\ref{propcenFI}, $FI(X,K)$ is a central algebra. By \cite[Remark~4]{AutK}, every multiplicative automorphism of $FI(X,K)$ is inner if and only if the first cohomology group $H^1(X,K^{\times})$ of $X$ with values in the multiplicative group $K^{\times}$ is trivial.

In this section, we present necessary and sufficient conditions for two involutions of the second kind on $FI(X,K)$ to be equivalent.

\begin{theorem}~\label{teogeralclassiinvolu}
Let $K$ be a field of characteristic different from $2$ and $X$ a connected poset such that every multiplicative automorphism of $FI(X,K)$ is inner. Let $\rho_1$ and $\rho_2$ be involutions of the second kind on $FI(X,K)$. Then $\rho_1$ and $\rho_2$ are equivalent if and only if there exist an automorphism $\alpha$ of $X$ and an inner automorphism $\Psi$ of $FI(X,K)$ such that $\rho_1 =\Psi\circ\widehat{\alpha}\circ\rho_2\circ \widehat{\alpha}^{-1}\circ \Psi^{-1}$.
\end{theorem}
\begin{proof}
The involutions $\rho_1$ and $\rho_2$ are equivalent if and only if there exists an automorphism $\Phi$ of $FI(X,K)$ such that $\Phi \circ \rho_2 = \rho_1 \circ \Phi$. By Theorem~\ref{DecoAutFIP}, $\Phi= \Psi \circ \widehat{\alpha}$ where $\Psi$ is an inner automorphism and $\alpha$ is an automorphism of $X$. Thus, $\rho_1$ and $\rho_2$ are equivalent if and only if $\rho_1 = \Psi \circ \widehat{\alpha} \circ  \rho_2 \circ \widehat{\alpha}^{-1}\circ \Psi^{-1}$.
\end{proof}

Note that in the theorem above, $\widehat{\alpha}\circ\rho_2\circ \widehat{\alpha}^{-1}$ is an involution on $FI(X,K)$ and it is equivalent to $\rho_1$, via inner automorphism. Because of that, we start by looking for necessary and sufficient conditions for two involutions of the second kind on $FI(X,K)$ to be equivalent via inner automorphisms.

\subsection{Classification via inner automorphisms}\label{Classification_inner}

\begin{proposition}\label{psi_de_um_lado_iff}
 Let $\rho_1$ and $\rho_2$ be involutions (of the first or second kind) on $FI(X,K)$. There exists $u\in U(FI(X,K))$ such that $\Psi_u\circ\rho_2=\rho_1$ if and only if $\rho_1$ and $\rho_2$ induce the same involution on $X$ and $\rho_1|_K=\rho_2|_K$.
\end{proposition}
\begin{proof}
Suppose there exists $u\in U(FI(X,K))$ such that $\Psi_u\circ\rho_2=\rho_1$. Let $\lambda_1$ and $\lambda_2$ be the involutions on $X$ induced by $\rho_1$ and $\rho_2$, respectively. For each $x\in X$, let $f_x\in U(FI(X,K))$ such that $\rho_2(e_x)=f_xe_{\lambda_2(x)}f_x^{-1}$. Then
$$\rho_1(e_x)=(\Psi_u\circ\rho_2)(e_x)=uf_xe_{\lambda_2(x)}f_x^{-1}u^{-1}=(uf_x)e_{\lambda_2(x)}(uf_x)^{-1}.$$
Thus $\lambda_1(x)=\lambda_2(x)$ for all $x\in X$ and so $\lambda_1=\lambda_2=:\lambda$. Denote $\rho_1|_K=\ast_1$ and $\rho_2|_K=\ast_2$. By Theorem~\ref{DecoInv} and \cite[Theorem~3.5]{BFS14}, there exist $v,w\in U(FI(X,K))$ such that $\rho_1=\Psi_v\circ \rho_{\lambda}^{\ast_1}$ and $\rho_2=\Psi_w\circ \rho_{\lambda}^{\ast_2}$. Thus, for all $k\in K$, it follows from \eqref{rolb_a} that
\begin{align*}
k^{\ast_1} & = \Psi_v(k^{\ast_1})=(\Psi_v\circ\rho_{\lambda}^{\ast_1})(k)= \rho_1(k) =  (\Psi_u\circ\rho_2)(k)=(\Psi_u\circ\Psi_w\circ \rho_{\lambda}^{\ast_2})(k)\\
           & = \Psi_{uw}(k^{\ast_2})= k^{\ast_2}.
\end{align*}
Hence $\rho_1|_K=\ast_1=\ast_2=\rho_2|_K$.

Conversely, suppose both involutions $\rho_1$ and $\rho_2$ induce the involution $\lambda$ on $X$ and $\rho_1|_K=\rho_2|_K=:\ast$. By Theorem~\ref{DecoInv} and \cite[Theorem~3.5]{BFS14}, there exist $v,w\in U(FI(X,K))$ such that $\rho_1=\Psi_v\circ \rho_{\lambda}^{\ast}$ and $\rho_2=\Psi_w\circ \rho_{\lambda}^{\ast}$. So
$$\rho_1=\Psi_v\circ \rho_{\lambda}^{\ast}=\Psi_v\circ\Psi_{w^{-1}}\circ\rho_2=\Psi_{vw^{-1}}\circ\rho_2.$$
\end{proof}

\begin{theorem}\label{necessary_inner}
Let $\rho_1$ and $\rho_2$ be involutions (of the first or second kind) on $FI(X,K)$. If there exists $u\in U(FI(X,K))$ such that $\Psi_u\circ\rho_2=\rho_1\circ\Psi_u$, then $\rho_1$ and $\rho_2$ induce the same involution on $X$ and $\rho_1|_K=\rho_2|_K$.
\end{theorem}
\begin{proof}
It follows directly from Lemma~\ref{inner_iff} and Proposition~\ref{psi_de_um_lado_iff}.
\end{proof}

It follows from Theorem~\ref{necessary_inner} that two involutions on $FI(X,K)$ that induce different involutions on $X$ or different involutions on $K$ are not equivalent via inner automorphisms. For this reason, we fix an involution $\lambda$ on $X$ and an involution $\ast$ of the second kind on $K$ and we are going to investigate when two involutions of the second kind on $FI(X,K)$ that induce $\lambda$ on $X$ and $\ast$ on $K$ are equivalent via inner automorphisms.

Firstly, we obtain an important characterization of the involutions on $FI(X,K)$ that induce $\lambda$ on $X$ and $\ast$ on $K$, under the assumptions of this section.

\begin{proposition}\label{dec_vsimetrico}
Let $K$ be a field of characteristic different from $2$ and $X$ a connected poset such that every multiplicative automorphism of $FI(X,K)$ is inner. Let $\lambda$ be an involution on $X$ and $\ast$ an involution of the second kind on $K$. If $\rho$ is an involution on $FI(X,K)$ that induces $\lambda$ on $X$ and $\ast$ on $K$, then there is a $\rho_{\lambda}^{\ast}$-symmetric $v\in U(FI(X,K))$ such that $\rho=\Psi_v\circ\rho_{\lambda}^{\ast}$.
\end{proposition}
\begin{proof}
By Theorem~\ref{DecoInv}, $\rho=\Psi_u\circ\rho_{\lambda}^{\ast}$ for some $u\in U(FI(X,K))$. Since $FI(X,K)$ is a central algebra, then, by Proposition~\ref{propAutInner} (ii), there is $k\in K^{\times}$ such that $\rho_{\lambda}^{\ast}(u)=ku$ and $\rho_{\lambda}^{\ast}(k)=k^{\ast}=k^{-1}$. By Proposition~\ref{k=a_ast_a-1}, there is $a\in K^{\times}$ such that $k=a^{\ast}a^{-1}$. Let $v=a^{-1}u\in U(FI(X,K))$. Then
$$\rho_{\lambda}^{\ast}(v)=\rho_{\lambda}^{\ast}(u)\rho_{\lambda}^{\ast}(a)^{-1}=ku(a^{\ast})^{-1}=kuk^{-1}a^{-1}=v.$$
Moreover, since $uv^{-1}=a\in K$, then $\Psi_u=\Psi_v$, by \cite[Proposition~2.2 (i)]{ER2}. Therefore, $\rho=\Psi_v\circ\rho_{\lambda}^{\ast}$.
\end{proof}

Let $(X_1, X_2, X_3)$ be a $\lambda$-decomposition of $X$. Let $\epsilon: X_3 \to K_0^{\times}$ be a map and consider $u_{\epsilon} \in U(FI(X,K))$ defined by
\begin{equation}\label{u_epsilon}
u_{\epsilon}(x,y) = \begin{cases}
0 & \text{if } x \neq y\\
1 & \text{if } x= y \in X\backslash X_3\\
\epsilon(x) & \text{if } x= y \in X_3
\end{cases}.
\end{equation}
Note that if $X_3=\emptyset$ or $\epsilon(x)=1$ for all $x\in X_3$, then $u_{\epsilon}=\delta$. For any $x,y\in X_3$,
\begin{align*}
 \rho_{\lambda}^{\ast}(u_{\epsilon})(x,y) & =[u_{\epsilon}(\lambda(y),\lambda(x))]^{\ast}  \\
   & = \begin{cases}
0 & \text{if } x \neq y\\
1 & \text{if } x= y \in X\backslash X_3\\
[\epsilon(x)]^\ast & \text{if } x= y \in X_3
\end{cases}\\
 & = u_{\epsilon}(x,y).
\end{align*}
Thus, $\rho_{\lambda}^{\ast}(u_{\epsilon})=u_{\epsilon}$. By Proposition~\ref{propAutInner}, $\rho_{\epsilon}:=\Psi_{u_{\epsilon}} \circ \rho_{\lambda}^{\ast}$ is an involution on $FI(X,K)$ such that $\rho_{\epsilon} = \rho_{\lambda}^{\ast}$ when $X_3 = \emptyset$ or $\epsilon(x)=1$ for all $x\in X_3$. Moreover, for any $x\in X$, $\rho_{\epsilon}(e_x)=u_{\epsilon}e_{\lambda(x)}u_{\epsilon}^{-1}$ by \eqref{rolb_e_x}, therefore $\rho_{\epsilon}$ induces the involution $\lambda$ on $X$. And, for $a\in K$, $\rho_{\epsilon}(a)=\Psi_{u_{\epsilon}}(a^{\ast})=a^{\ast}$, by \eqref{rolb_a}. Thus, $\rho_{\epsilon}|_K=\ast$. In particular, $\rho_{\epsilon}$ is of the second kind, since $\ast$ is of the second kind. Hence we have the following proposition.

\begin{proposition}
Let $(X_1, X_2, X_3)$ be a $\lambda$-decomposition of $X$. Let $\epsilon: X_3 \to K_0^{\times}$ be a map and $u_{\epsilon} \in U(FI(X,K))$ as in \eqref{u_epsilon}. Then
\begin{equation}\label{rho_epsilon}
 \rho_{\epsilon}=\Psi_{u_{\epsilon}} \circ \rho_{\lambda}^{\ast}
\end{equation}
is an involution of the second kind on $FI(X,K)$ that induces $\lambda$ on $X$ and $\ast$ on $K$. Moreover, $\rho_{\epsilon} = \rho_{\lambda}^{\ast}$ when $X_3 = \emptyset$ or $\epsilon(x)=1$ for all $x\in X_3$.
\end{proposition}

\begin{lemma} \label{lemaiffsimetrico}
Let $u\in U(FI(X,K))$ be $\rho_{\epsilon}$-symmetric. Then $u=v\rho_{\epsilon}(v)$ for some $v\in U(FI(X,K))$ if and only if $u(x,x) \in K_1$, for all $x \in X_3$.
\end{lemma}
\begin{proof}
If there exists $v\in U(FI(X,K))$ such that $u=v\rho_{\epsilon}(v)$, then for all $x \in X_3$,
\begin{align*}
  u(x,x) & = v(x,x)\rho_{\epsilon}(v)(x,x)= v(x,x)(u_{\epsilon}\rho_{\lambda}^{\ast}(v)u_{\epsilon}^{-1})(x,x)\\
         & = v(x,x)u_{\epsilon}(x,x)[v(\lambda(x),\lambda(x))]^{\ast}u_{\epsilon}^{-1}(x,x)\\
         & = v(x,x)[v(x,x)]^{\ast}u_{\epsilon}(x,x)[u_{\epsilon}(x,x)]^{-1}\\
         & =v(x,x)[v(x,x)]^{\ast}\in K_1.
\end{align*}

Conversely, suppose that for each $x \in X_3$ there exists $a_x\in K^{\times}$ such that $u(x,x)=a_xa_x^{\ast}$. For $x \leq y$ in $X$ there are six possibilities:
\begin{enumerate}[(1)]
\item $x,y\in X_1$;
\item  $x,y\in X_2$;
\item $x \in X_1$ and $y \in X_2$;
\item $x \in X_1$  and  $y \in X_3$;
\item $x \in X_3$ and $y \in X_2$;
\item $x = y \in X_3$.
\end{enumerate}
So we define $v \in I(X,K)$ as follows, for each pair $x\leq y$ in $X$:
$$v(x,y)=\begin{cases}
\delta_{xy} & \text{if } x , y \in X_1\\
u(x,y) & \text{if } x , y \in X_2\\
u(x,y)/2 & \text{if } x \in X_1 \text{ and } y \in X_2\\
0 &\text{if } x \in X_1 \text{ and } y \in X_3\\
u(x,y) & \text{if } x \in X_3 \text{ and } y \in X_2\\
a_x & \text{if } x = y \in X_3
\end{cases}.$$
If $[r,s]\subseteq [x,y]$ is such that $r\neq s$ and $v(r,s)\neq 0$, then $r,s\in X_2$, or $r\in X_1$ and $s\in X_2$, or $r\in X_3$ and $s\in X_2$. In any of these cases, we get $u(r,s)\neq 0$. Since $u\in FI(X,K)$, the number of such intervals $[r,s]$ must be finite, hence $v\in FI(X,K)$. Moreover, clearly $v$ is invertible.

Let $x\leq y$ in $X$. We are going to prove that $u(x,y) = (v\rho_{\epsilon}(v))(x,y)$ for each of the six possibilities listed above.

(1. $x,y \in X_1$) In this case, $\lambda(x), \lambda(y) \in X_2$ and if $t \in X$ is such that $x \leq t \leq y$, then $t \in X_1$. Thus, $v(x,t) = 0$ for all $x<t\leq y$, so
\begin{align*}
(v\rho_{\epsilon}(v))(x,y) & = \sum_{x\leq t\leq y}v(x,t)\rho_{\epsilon}(v)(t,y)=v(x,x)\rho_{\epsilon}(v)(x,y)\\
                           & = u_{\epsilon}(x,x)\rho_{\lambda}^{\ast}(v)(x,y)u_{\epsilon}^{-1}(y,y)=[v(\lambda(y), \lambda(x))]^{\ast}=[u(\lambda(y), \lambda(x))]^{\ast}\\
                           & =\rho_{\lambda}^{\ast}(u)(x,y)=u_{\epsilon}(x,x)\rho_{\lambda}^{\ast}(u)(x,y)u_{\epsilon}^{-1}(y,y)=\rho_{\epsilon}(u)(x,y)\\
                           & = u(x,y).
\end{align*}

(2. $x,y \in X_2$) We have $\lambda(x), \lambda(y) \in X_1$ and if $t \in X$ is such that $x\leq t \leq y$, then $t \in X_2$. So if $x\leq t < y$, then $v(\lambda(y), \lambda(t)) = 0$, therefore
\begin{align*}
(v\rho_{\epsilon}(v))(x,y) & = \sum_{x\leq t\leq y}v(x,t)(u_{\epsilon}\rho_{\lambda}^{\ast}(v)u_{\epsilon}^{-1})(t,y)\\
                           & = \sum_{x\leq t\leq y}v(x,t)u_{\epsilon}(t,t)[v(\lambda(y),\lambda(t))]^{\ast}u_{\epsilon}^{-1}(y,y)\\
                           & = v(x,y)[v(\lambda(y), \lambda(y))]^{\ast}\\
                           & = u(x,y).
\end{align*}

(3. $x\in X_1$, $y\in X_2$) In this case, if $x\leq t\leq y$, then $t \in X_1 \cup X_2 \cup X_3$. If $x<t<y$ and $t\in X_1\cup X_3$, then $v(x,t)=0$. And if $x<t<y$ and $t \in X_2$, then $\lambda(t) \in X_1$, therefore,
$$\rho_{\epsilon}(v)(t,y)= u_{\epsilon}(t,t)[v(\lambda(y),\lambda(t))]^{\ast}u_{\epsilon}^{-1}(y,y)=[v(\lambda(y),\lambda(t))]^{\ast}=0.$$
Thus,
 \begin{align*}
 (v\rho_{\epsilon}(v))(x,y) & = \sum_{x\leq t\leq y}v(x,t)\rho_{\epsilon}(v)(t,y)\\
                            & = v(x,x)\rho_{\epsilon}(v)(x,y) + v(x,y)\rho_{\epsilon}(v)(y,y)\\
                            & = u_{\epsilon}(x,x)[v(\lambda(y),\lambda(x))]^{\ast}u_{\epsilon}^{-1}(y,y)\\
                            & \quad  +\frac{u(x,y)}{2}u_{\epsilon}(y,y)[v(\lambda(y),\lambda(y))]^{\ast}u_{\epsilon}^{-1}(y,y)\\
                            & = \left[\frac{u(\lambda(y),\lambda(x))}{2}\right]^{\ast}+\frac{u(x,y)}{2}\\
                            & = \frac{\rho_{\lambda}^{\ast}(u)(x,y)}{2}+\frac{u(x,y)}{2}\\
                            & = \frac{u_{\epsilon}(x,x)\rho_{\lambda}^{\ast}(u)(x,y)u_{\epsilon}^{-1}(y,y)}{2}+\frac{u(x,y)}{2}\\
                            & = \frac{\rho_{\epsilon}(u)(x,y)}{2}+\frac{u(x,y)}{2}=\frac{u(x,y)}{2}+\frac{u(x,y)}{2}\\
                            & = u(x,y).
\end{align*}

(4. $x\in X_1$, $y\in X_3$) If $x\leq t \leq y$, then $t \in X_1 \cup X_3$, and $t\in X_3$ implies $t=y$. Then
 \begin{align*}
 (v\rho_{\epsilon}(v))(x,y) & = \sum_{x\leq t\leq y}v(x,t)\rho_{\epsilon}(v)(t,y)\\
                            & = v(x,x)\rho_{\epsilon}(v)(x,y) + v(x,y)\rho_{\epsilon}(v)(y,y)\\
                            & = u_{\epsilon}(x,x)[v(\lambda(y),\lambda(x))]^{\ast}u_{\epsilon}^{-1}(y,y)+0\\
                            & = [v(y,\lambda(x))]^{\ast}[\epsilon(y)]^{-1} = [u(y,\lambda(x))]^{\ast}[\epsilon(y)]^{-1}\\
                            & = u_{\epsilon}(x,x)[u(\lambda(y),\lambda(x))]^{\ast}u_{\epsilon}^{-1}(y,y)u_{\epsilon}(y,y)[\epsilon(y)]^{-1}\\
                            & = \rho_{\epsilon}(u)(x,y)\epsilon(y)[\epsilon(y)]^{-1}\\
                            & = u(x,y).
\end{align*}

(5. $x\in X_3$, $y\in X_2$) If $x\leq t \leq y$, then $t \in X_2 \cup X_3$, and $t\in X_3$ implies $t=x$. Thus
 \begin{align*}
 (v\rho_{\epsilon}(v))(x,y) & = \sum_{x\leq t\leq y}v(x,t)\rho_{\epsilon}(v)(t,y)\\
                            & = v(x,x)u_{\epsilon}(x,x)[v(\lambda(y),\lambda(x))]^{\ast}u_{\epsilon}^{-1}(y,y)\\
                            & \quad + \sum_{x<t\leq y}v(x,t)u_{\epsilon}(t,t)[v(\lambda(y),\lambda(t))]^{\ast}u_{\epsilon}^{-1}(y,y)\\
                            & = 0+v(x,y)u_{\epsilon}(y,y)[v(\lambda(y),\lambda(y))]^{\ast}u_{\epsilon}^{-1}(y,y)\\
                            & = u(x,y).
\end{align*}

(6. $x,y \in X_3$) In this case, $x=y$ and so
\begin{align*}
(v\rho_{\epsilon}(v))(x,y) & = (v\rho_{\epsilon}(v))(x,x)=v(x,x)\rho_{\epsilon}(v)(x,x)\\
                           & = a_xu_{\epsilon}(x,x)[v(\lambda(x),\lambda(x))]^{\ast}u_{\epsilon}^{-1}(x,x)\\
                           & = a_x[v(\lambda(x),\lambda(x))]^{\ast}=a_x[v(x,x)]^{\ast}\\
                           & = a_xa_x^{\ast}=u(x,x)\\
                           & = u(x,y).
\end{align*}

We thus conclude that $u=v\rho_{\epsilon}(v)$.
\end{proof}

Since $\rho_{\epsilon} = \rho_{\lambda}^{\ast}$ when $X_3 = \emptyset$, the following corollary is immediate.

\begin{corollary} \label{cor_rho_simetrico}
Let $u\in U(FI(X,K))$ such that $\rho_{\lambda}^{\ast}(u)=u$. If $X_3 = \emptyset$, then $u=v\rho_{\lambda}^{\ast}(v)$ for some $v\in U(FI(X,K))$.
\end{corollary}

In the case where $X_3=\emptyset$ we have the following result.

\begin{theorem}~\label{x_3vazio}
Let $K$ be a field of characteristic different from $2$ and $X$ a connected poset such that every multiplicative automorphism of $FI(X,K)$ is inner. Let $\ast$ be an involution of the second kind on $K$ and $\lambda$ an involution on $X$ such that $\lambda(x)\neq x$ for all $x\in X$. Then every involution on $FI(X,K)$ that induces $\lambda$ on $X$ and $\ast$ on $K$ is equivalent to $\rho_{\lambda}^{\ast}$ (via inner automorphisms).
\end{theorem}
\begin{proof}
Let $\rho$ be an involution on $FI(X,K)$ that induces $\lambda$ on $X$ and $\ast$ on $K$. By Proposition~\ref{dec_vsimetrico}, there is $u\in U(FI(X,K))$ such that $\rho_{\lambda}^{\ast}(u)=u$ and $\rho=\Psi_u\circ\rho_{\lambda}^{\ast}$. By Corollary~\ref{cor_rho_simetrico}, there is $v\in U(FI(X,K))$ such that $u=v\rho_{\lambda}^{\ast}(v)$. Thus, by Lemma~\ref{inner_iff}, $\rho$ is equivalent to $\rho_{\lambda}^{\ast}$ via inner automorphisms.
\end{proof}

From now on we will analyze the case where $X_3\neq \emptyset$.

Note that if $u\in FI(X,K)$ is $\rho_\lambda^\ast$-symmetric, then for each $x\in X_3$, $u(x,x)=\rho_\lambda^\ast(u)(x,x)=[u(\lambda(x),\lambda(x))]^\ast=[u(x,x)]^\ast$. Thus, for each $v \in U(FI(X,K))$ such that $\rho_\lambda^\ast(v)=v$ we can define a map $\epsilon_v : X_3 \to K_0^{\times}$ by $\epsilon_v(x)= v(x,x)$.

\begin{lemma}\label{rho_equiv_rho_epsilon}
Let $\rho$ be an involution on $FI(X,K)$ that induces $\lambda$ on $X$ and $\ast$ on $K$. Then there is a $\rho_\lambda^\ast$-symmetric $v \in U(FI(X,K))$ such that $\rho$ is equivalent to $\rho_{\epsilon_v}$, via inner automorphisms.
\end{lemma}
\begin{proof}
By Proposition~\ref{dec_vsimetrico}, there is $v\in U(FI(X,K))$ such that $\rho_{\lambda}^{\ast}(v)=v$ and $\rho=\Psi_v\circ\rho_{\lambda}^{\ast}$. Since $\rho_{\epsilon_v}=\Psi_{u_{\epsilon_v}}\circ\rho_{\lambda}^{\ast}$, by \eqref{rho_epsilon}, then $\rho=\Psi_v\circ\Psi_{u_{\epsilon_v}^{-1}}\circ\rho_{\epsilon_v}=\Psi_{vu_{\epsilon_v}^{-1}}\circ\rho_{\epsilon_v}$. Now
$$
 \rho_{\epsilon_v}(vu_{\epsilon_v}^{-1}) = \rho_{\epsilon_v}(u_{\epsilon_v})^{-1} \rho_{\epsilon_v}(v)=u_{\epsilon_v}^{-1}\Psi_{u_{\epsilon_v}}(\rho_{\lambda}^{\ast}(v))
                                         = u_{\epsilon_v}^{-1}\Psi_{u_{\epsilon_v}}(v)= vu_{\epsilon_v}^{-1}.
$$
Moreover, for each $x\in X_3$,
$$(vu_{\epsilon_v}^{-1})(x,x)=v(x,x)[u_{\epsilon_v}(x,x)]^{-1}=\epsilon_v(x)[\epsilon_v(x)]^{-1}=1\in K_1.$$
Thus, by Lemma~\ref{lemaiffsimetrico}, there exists $w\in U(FI(X,K))$ such that $vu_{\epsilon_v}^{-1}=w\rho_{\epsilon_v}(w)$. It follows that $\rho=\Psi_{w\rho_{\epsilon_v}(w)}\circ\rho_{\epsilon_v}$ and hence $\rho$ is equivalent to $\rho_{\epsilon_v}$, via inner automorphisms, by Lemma~\ref{inner_iff}.
\end{proof}

It follows from lemma above that every involution $\rho$ on $FI(X,K)$ that induces $\lambda$ on $X$ and $\ast$ on $K$ is equivalent, via inner automorphisms, to $\rho_\epsilon$ for some map $\epsilon: X_3 \to K_0^{\times}$. Thus, to get a classification of involutions via inner automorphisms it is enough to find out what conditions $\epsilon_1,\epsilon_2: X_3 \to K_0^{\times}$ must satisfy for $\rho_{\epsilon_1}$ and $\rho_{\epsilon_2}$ to be equivalent via inner automorphisms.

\begin{lemma}\label{rho_es_equiv1}
Let $\epsilon_1,\epsilon_2: X_3 \to K_0^{\times}$ be maps. The involutions $\rho_{\epsilon_1}$ and $\rho_{\epsilon_2}$ are equivalent via inner automorphisms if and only if there is $k\in K_0^{\times}$ such that, for each $x\in X_3$, $\epsilon_1(x)=ka_x\epsilon_2(x)$ for some $a_x\in K_1$.
\end{lemma}
\begin{proof}
If $\rho_{\epsilon_1}$ and $\rho_{\epsilon_2}$ are equivalent via inner automorphisms, then, by Lemma~\ref{inner_iff}, there exists $v\in U(FI(X,K))$ such that $\rho_{\epsilon_1}=\Psi_{v\rho_{\epsilon_2}(v)}\circ\rho_{\epsilon_2}$. On the other hand,
\begin{equation}\label{rho_e1_rho_e2}
  \rho_{\epsilon_1}=\Psi_{u_{\epsilon_1}} \circ \rho_{\lambda}^{\ast}=\Psi_{u_{\epsilon_1}} \circ \Psi_{u_{\epsilon_2}^{-1}}\circ\rho_{\epsilon_2}=\Psi_{u_{\epsilon_1}u_{\epsilon_2}^{-1}}\circ\rho_{\epsilon_2},
\end{equation}
by \eqref{rho_epsilon}. Thus $\Psi_{v\rho_{\epsilon_2}(v)}=\Psi_{u_{\epsilon_1}u_{\epsilon_2}^{-1}}$, so by \cite[Proposition~2.2 (i)]{ER2} and Proposition~\ref{propcenFI}, there is $k\in K^{\times}$ such that $v\rho_{\epsilon_2}(v)=ku_{\epsilon_1}u_{\epsilon_2}^{-1}$. By Lemma~\ref{lemaiffsimetrico}, for each $x\in X_3$, there is $a_x\in K_1$ such that $(v\rho_{\epsilon_2}(v))(x,x)=a_x$. Therefore, by \eqref{u_epsilon},
\begin{equation}\label{a_e1_e2}
  a_x=ku_{\epsilon_1}(x,x)[u_{\epsilon_2}(x,x)]^{-1}=k\epsilon_1(x)[\epsilon_2(x)]^{-1}.
\end{equation}
Since $a_x\in K_1\subseteq K_0^\times$ and $\epsilon_1(x),\epsilon_2(x)\in K_0^\times$, then $k\in K_0^\times$. Thus, by \eqref{a_e1_e2}, $\epsilon_1(x)=k^{-1}a_x\epsilon_2(x)$ with $k^{-1}\in K_0^\times$ and $a_x\in K_1$, for each $x\in X_3$.

Conversely, suppose there is $k\in K_0^{\times}$ such that, for each $x\in X_3$, $\epsilon_1(x)=ka_x\epsilon_2(x)$ for some $a_x\in K_1$. By \eqref{rho_e1_rho_e2} and \cite[Proposition~2.2 (i)]{ER2},
\begin{equation}\label{k_rho_e1_rho_e2}
\rho_{\epsilon_1}=\Psi_{k^{-1}u_{\epsilon_1}u_{\epsilon_2}^{-1}}\circ\rho_{\epsilon_2}.
\end{equation}
Since $u_{\epsilon_1}u_{\epsilon_2}^{-1}$ and $k$ are $\rho_{\epsilon_2}$-symmetric, then so is $k^{-1}u_{\epsilon_1}u_{\epsilon_2}^{-1}$. Moreover, for any $x\in X_3$,
\begin{align*}
(k^{-1}u_{\epsilon_1}u_{\epsilon_2}^{-1})(x,x) & = k^{-1}u_{\epsilon_1}(x,x)[u_{\epsilon_2}(x,x)]^{-1}\\
                                               & = k^{-1}\epsilon_1(x)[\epsilon_2(x)]^{-1}=a_x\in K_1.
\end{align*}
By Lemma~\ref{lemaiffsimetrico}, there is $v\in U(FI(X,K))$ such that $k^{-1}u_{\epsilon_1}u_{\epsilon_2}^{-1}=v\rho_{\epsilon_2}(v)$. Thus, by \eqref{k_rho_e1_rho_e2}, $\rho_{\epsilon_1}=\Psi_{v\rho_{\epsilon_2}(v)}\circ\rho_{\epsilon_2}$, therefore $\rho_{\epsilon_1}$ and $\rho_{\epsilon_2}$ are equivalent via inner automorphisms, by Lemma~\ref{inner_iff}.
\end{proof}

Let $\pi:K_0^{\times}\to K_0^{\times}/K_1$ be the canonical homomorphism. For any map $\epsilon: X_3 \to K_0^{\times}$, let $\theta_{\epsilon}=\pi\circ\epsilon$. We define on the set $(K_0^{\times}/K_1)^{X_3}$ of all maps from $X_3$ to $K_0^{\times}/K_1$ the following equivalence relation: $\theta_1, \theta_2\in (K_0^{\times}/K_1)^{X_3}$ are equivalent if there exists $g\in K_0^{\times}/K_1$ such that $\theta_1=L_g\circ\theta_2$, where $L_g:K_0^{\times}/K_1\to K_0^{\times}/K_1$ is the left multiplication by $g$.

\begin{remark}\label{theta=theta_epsilon}
For each $\theta\in (K_0^{\times}/K_1)^{X_3}$ there is $\epsilon: X_3 \to K_0^{\times}$ such that $\theta=\theta_{\epsilon}$. Indeed, for each $x\in X_3$ take an $a_x\in K_0^{\times}$ such that $\theta(x)=a_xK_1$ and define $\epsilon : X_3 \to K_0^{\times}$ by $\epsilon(x)=a_x$.
\end{remark}

\begin{corollary}\label{rho_es_equiv2}
Let $\epsilon_1,\epsilon_2: X_3 \to K_0^{\times}$ be maps. The involutions $\rho_{\epsilon_1}$ and $\rho_{\epsilon_2}$ are equivalent via inner automorphisms if and only if $\theta_{\epsilon_1}$ and $\theta_{\epsilon_2}$ are equivalent.
\end{corollary}
\begin{proof}
If $\rho_{\epsilon_1}$ and $\rho_{\epsilon_2}$ are equivalent via inner automorphisms, then there is $k\in K_0^{\times}$ such that, for each $x\in X_3$, $\epsilon_1(x)=ka_x\epsilon_2(x)$ for some $a_x\in K_1$, by Lemma~\ref{rho_es_equiv1}. Thus, for any $x\in X_3$, we have
\begin{align*}
\theta_{\epsilon_1}(x) & = \pi(\epsilon_1(x))=\epsilon_1(x)K_1=(ka_x\epsilon_2(x))K_1 = (kK_1)(\epsilon_2(x)K_1)\\
                       & = (kK_1)\pi(\epsilon_2(x))=(kK_1)\theta_{\epsilon_2}(x)=(L_{kK_1}\circ \theta_{\epsilon_2})(x).
\end{align*}
Therefore, $\theta_{\epsilon_1}=L_{kK_1}\circ \theta_{\epsilon_2}$.

Conversely, if there is $g=kK_1\in K_0^{\times}/K_1$ such that $\theta_{\epsilon_1}=L_{g}\circ \theta_{\epsilon_2}$, then for all $x\in X_3$,
$$\epsilon_1(x)K_1=\theta_{\epsilon_1}(x)=(L_g\circ\pi\circ\epsilon_2)(x)=(kK_1)(\epsilon_2(x)K_1)=(k\epsilon_2(x))K_1.$$
So for each $x\in X_3$ there is $a_x\in K_1$ such that $\epsilon_1(x)=ka_x\epsilon_2(x)$. By Lemma~\ref{rho_es_equiv1}, $\rho_{\epsilon_1}$ and $\rho_{\epsilon_2}$ are equivalent via inner automorphisms.
\end{proof}

If $\rho$ is an involution on $FI(X,K)$ that induces $\lambda$ on $X$ and $\ast$ on $K$, then there exists a map $\epsilon: X_3 \to K_0^{\times}$ such that $\rho$ is equivalent to $\rho_\epsilon$, via inner automorphisms, by Lemma~\ref{rho_equiv_rho_epsilon}. Thus, by Corollary~\ref{rho_es_equiv2}, $\rho$ induces on $(K_0^{\times}/K_1)^{X_3}$ only one equivalence class, namely, the equivalence class that contains $\theta_{\epsilon}$. Therefore, by Remark~\ref{theta=theta_epsilon}, the number of equivalence classes of involutions on $FI(X,K)$ that induce $\lambda$ on $X$ and $\ast$ on $K$, via inner automorphisms, is equal to the number of equivalence classes on the set $(K_0^{\times}/K_1)^{X_3}$. So we have the following result.

\begin{theorem}\label{x_3naovazio}
Let $K$ be a field of characteristic different from $2$ and $X$ a connected poset such that every multiplicative automorphism of $FI(X,K)$ is inner. Let $\ast$ be an involution of the second kind on $K$ and $\lambda$ an involution on $X$ such that $\lambda(x)=x$ for some $x\in X$. Let $\rho_1$ and $\rho_2$ be involutions on $FI(X,K)$ that induce $\lambda$ on $X$ and $\ast$ on $K$. Then $\rho_1$ and $\rho_2$ are equivalent via inner automorphisms if and only if they induce the same equivalence class on $(K_0^{\times}/K_1)^{X_3}$.
\end{theorem}

Since the cardinality of any equivalence class on $(K_0^{\times}/K_1)^{X_3}$ is equal to $|K_0^{\times}/K_1|$, there are $|K_0^{\times}/K_1|^{|X_3|-1}$ classes of involutions on $FI(X,K)$ that induce $\lambda$ on $X$ and $\ast$ on $K$, via inner automorphisms. Thus, in the particular case when $|X_3|=1$, we have the following result.

\begin{theorem}\label{x_3unitario}
Let $K$ be a field of characteristic different from $2$ and $X$ a connected poset such that every multiplicative automorphism of $FI(X,K)$ is inner. Let $\ast$ be an involution of the second kind on $K$ and $\lambda$ an involution on $X$ such that $\lambda(x)=x$ for exatly one $x\in X$. Then every involution on $FI(X,K)$ that induces $\lambda$ on $X$ and $\ast$ on $K$ is equivalent to $\rho_{\lambda}^{\ast}$ (via inner automorphisms).
\end{theorem}

\subsection{General classification}\label{Genereal_classification}

Finally, we will use Theorem~\ref{teogeralclassiinvolu} to obtain the classification of involutions of the second kind on $FI(X,K)$ from the classification via inner automorphisms.

\begin{proposition}\label{ala-1}
Let $\rho$ be an involution of the second kind on $FI(X,K)$ that induces the involutions $\lambda$ on $X$ and $\ast$ on $K$, an let $\alpha$ be an automorphism of $X$. Then the involution $\widehat{\alpha}\circ\rho\circ \widehat{\alpha}^{-1}$ induces the involution $\alpha\circ\lambda\circ\alpha^{-1}$ on $X$ and the involution $\ast$ on $K$.
\end{proposition}
\begin{proof}
For each $x\in X$, let $f_x\in U(FI(X,K))$ such that $\rho(e_x)=f_xe_{\lambda(x)}f_x^{-1}$. Then
\begin{align*}
  (\widehat{\alpha}\circ\rho\circ \widehat{\alpha}^{-1})(e_x) & = (\widehat{\alpha}\circ\rho)(e_{\alpha^{-1}(x)})\\
  & = \widehat{\alpha}(f_{\alpha^{-1}(x)}e_{\lambda(\alpha^{-1}(x))}f^{-1}_{\alpha^{-1}(x)})\\
  & = \widehat{\alpha}(f_{\alpha^{-1}(x)})e_{(\alpha\circ\lambda\circ\alpha^{-1})(x)}[\widehat{\alpha}(f_{\alpha^{-1}(x)})]^{-1}
\end{align*}
whence $\widehat{\alpha}\circ\rho\circ \widehat{\alpha}^{-1}$ induces $\alpha\circ\lambda\circ\alpha^{-1}$ on $X$. Now, since $FI(X,K)$ is central and $\widehat{\alpha}$ is $K$-linear, then, for all $a\in K$,
$$(\widehat{\alpha}\circ\rho\circ \widehat{\alpha}^{-1})(a)=(\widehat{\alpha}\circ\rho)(a)=\widehat{\alpha}(a^\ast)=a^\ast.$$
\end{proof}

As in \cite[p.1953]{BFS11}, we consider the following equivalence relation $\sim$ on the set of all involutions on $X$: $\lambda$ and $\mu$ are equivalent if there exists an automorphism $\alpha$ of $X$ such that $\alpha\circ \lambda = \mu \circ \alpha$.

\begin{corollary}~\label{cor_general}
Let $K$ be a field of characteristic different from $2$ and $X$ a connected poset such that every multiplicative automorphism of $FI(X,K)$ is inner.
\begin{enumerate}
  \item[(i)] If the involutions $\rho_1$ and $\rho_2$ of the second kind on $FI(X,K)$ are equivalent, then they induce equivalent involutions on $X$ and $\rho_1|_K=\rho_2|_K$.
  \item[(ii)] Let $\lambda_1$ and $\lambda_2$ be equivalent involutions on $X$, $\alpha$ an automorphism of $X$ such that $\alpha\circ \lambda_2 = \lambda_1 \circ \alpha$, and $\ast$ an involution of the second kind on $K$. If $\rho$ is an involution on $FI(X,K)$ that induces $\lambda_2$ on $X$ and $\ast$ on $K$, then $\rho$ is equivalent to the involution $\widehat{\alpha}\circ\rho\circ \widehat{\alpha}^{-1}$ which induces $\lambda_1$ on $X$ and $\ast$ on $K$.
\end{enumerate}
\end{corollary}
\begin{proof}
(i) Let $\lambda_1$ and $\lambda_2$ be the involutions on $X$ induced by the involutions $\rho_1$ and $\rho_2$, respectively. If $\rho_1$ and $\rho_2$ are equivalent, then there exists an automorphism $\alpha$ of $X$ such that $\rho_1$ and $\widehat{\alpha}\circ\rho_2\circ \widehat{\alpha}^{-1}$ are equivalent via inner automorphism, by Theorem~\ref{teogeralclassiinvolu}. Thus, by Theorem~\ref{necessary_inner} and Proposition~\ref{ala-1}, $\lambda_1=\alpha\circ\lambda_2\circ\alpha^{-1}$ and $\rho_1|_K=\rho_2|_K$.

(ii) It is clear that $\widehat{\alpha}\circ\rho\circ \widehat{\alpha}^{-1}$ is an involution on $FI(X,K)$ which is equivalent to $\rho$. Moreover, $\widehat{\alpha}\circ\rho\circ \widehat{\alpha}^{-1}$ induces $\lambda_1$ on $X$ and $\ast$ on $K$, by Proposition~\ref{ala-1}.
\end{proof}

\begin{remark}\label{al=la}
Let $\lambda_1,\lambda_2$ be equivalent involutions on $X$, and $X_3^{\lambda_i}=\{x\in X : \lambda_i(x)=x\}$, $i=1,2$. Let $\alpha$ be an automorphism of $X$ such that $\alpha\circ\lambda_2=\lambda_1\circ\alpha$. We have
$$\lambda_2(x)=x\Leftrightarrow \alpha(\lambda_2(x))=\alpha(x)\Leftrightarrow \lambda_1(\alpha(x))=\alpha(x).$$
Thus, $|X_3^{\lambda_1}|=|X_3^{\lambda_2}|$, and given $\epsilon:X_3^{\lambda_2}\to K_0^{\times}$ we have a well-defined map $\epsilon\circ\alpha^{-1}:X_3^{\lambda_1}\to K_0^{\times}$.
\end{remark}

\begin{theorem}\label{x_3lessthanorequalto1}
Let $K$ be a field of characteristic different from $2$ and $X$ a connected poset such that every multiplicative automorphism of $FI(X,K)$ is inner. Let $\lambda_1,\lambda_2$ be involutions on $X$ such that $|X_3^{\lambda_1}|=|X_3^{\lambda_2}|\leq 1$. Let $\rho_1$ and $\rho_{2}$ be involutions of the second kind on $FI(X,K)$ that induce $\lambda_1$ and $\lambda_2$ on $X$, respectively. Then $\rho_{1}$ and $\rho_{2}$ are equivalent if and only if $\lambda_1$ and $\lambda_2$ are equivalent and $\rho_1|_K=\rho_2|_K$.
\end{theorem}
\begin{proof}
If $\rho_{1}$ and $\rho_{2}$ are equivalent, then $\lambda_1$ and $\lambda_2$ are equivalent and $\rho_1|_K=\rho_2|_K$, by Corollary~\ref{cor_general} (i).

Conversely, suppose $\lambda_1$ and $\lambda_2$ are equivalent and $\rho_1|_K=\rho_2|_K=:\ast$. By Theorems~\ref{x_3vazio} and \ref{x_3unitario}, $\rho_i$ is equivalent (via inner automorphism) to $\rho_{\lambda_i}^\ast$, $i=1,2$. Let $\alpha$ be an automorphism of $X$ such that $\alpha\circ\lambda_2=\lambda_1\circ\alpha$. By Corollary~\ref{cor_general} (ii), $\rho_{\lambda_2}^\ast$ is equivalent to the involution $\widehat{\alpha}\circ\rho_{\lambda_2}^\ast\circ \widehat{\alpha}^{-1}$. However, for any $f\in FI(X,K)$ and $x\leq y$ in $X$, we have
\begin{align*}
  (\widehat{\alpha}\circ\rho_{\lambda_2}^\ast\circ \widehat{\alpha}^{-1})(f)(x,y) & = (\rho_{\lambda_2}^\ast\circ \widehat{\alpha}^{-1})(f)(\widehat{\alpha}^{-1}(x),\widehat{\alpha}^{-1}(y))\\
  & = [\widehat{\alpha}^{-1}(f)(\lambda_2(\alpha^{-1}(y)),\lambda_2(\alpha^{-1}(x))]^\ast \\
  & = [f((\alpha\circ\lambda_2\circ\alpha^{-1})(y),(\alpha\circ\lambda_2\circ\alpha^{-1})(x))]^\ast\\
  & = [f(\lambda_1(y),\lambda_1(x))]^\ast = \rho_{\lambda_1}^\ast(f)(x,y),
\end{align*}
that is, $\widehat{\alpha}\circ\rho_{\lambda_2}^\ast\circ \widehat{\alpha}^{-1}=\rho_{\lambda_1}^\ast$. Therefore, $\rho_{\lambda_1}^\ast$ and $\rho_{\lambda_2}^\ast$ are equivalent and hence $\rho_{1}$ and $\rho_{2}$ are equivalent.
\end{proof}

The next theorem completes the classification of involutions of the second kind on $FI(X,K)$ (see Lemma~\ref{rho_equiv_rho_epsilon}, Corollary~\ref{rho_es_equiv2}, Theorem~\ref{x_3naovazio}, Corollary~\ref{cor_general} and Theorem~\ref{x_3lessthanorequalto1}). To prove it, we need a lemma.

\begin{lemma}\label{conjugado_rho_epsilon}
Let $K$ be a field of characteristic different from $2$ and $X$ a connected poset such that every multiplicative automorphism of $FI(X,K)$ is inner. Let $\ast$ be an involution of the second kind on $K$, $\lambda_1,\lambda_2$ equivalent involutions on $X$, and $X_3^{\lambda_i}=\{x\in X : \lambda_i(x)=x\}$, $i=1,2$. Let $\alpha$ be an automorphism of $X$ such that $\alpha\circ\lambda_2=\lambda_1\circ\alpha$. If $\rho_{\epsilon_2}=\Psi_{u_{\epsilon_2}} \circ \rho_{\lambda_2}^{\ast}$, then $\widehat{\alpha}\circ\rho_{\epsilon_2}\circ \widehat{\alpha}^{-1}=\rho_{\epsilon_2\circ \alpha^{-1}}$.
\end{lemma}
\begin{proof}
By \eqref{u_epsilon} and Remark~\ref{al=la}, we have
\begin{align*}
u_{\epsilon_2}(\alpha^{-1}(x),\alpha^{-1}(x)) & = \begin{cases}
1 & \text{if } \alpha^{-1}(x) \in X\backslash X_3^{\lambda_2}\\
\epsilon_2(\alpha^{-1}(x)) & \text{if } \alpha^{-1}(x) \in X_3^{\lambda_2}
\end{cases}\\
 & = \begin{cases}
1 & \text{if } x \in X\backslash X_3^{\lambda_1}\\
(\epsilon_2\circ\alpha^{-1})(x) & \text{if } x \in X_3^{\lambda_1}
\end{cases}\\
 & = u_{\epsilon_2\circ \alpha^{-1}}(x,x).
\end{align*}
Therefore, for any $f\in FI(X,K)$ and $x\leq y$ in $X$, we have
\begin{align*}
  (\widehat{\alpha}\circ\rho_{\epsilon_2}\circ \widehat{\alpha}^{-1})(f)(x,y) & = (\rho_{\epsilon_2}\circ \widehat{\alpha}^{-1})(f)(\alpha^{-1}(x),\alpha^{-1}(y))\\
  & = (\Psi_{u_{\epsilon_2}} \circ \rho_{\lambda_2}^{\ast}\circ\widehat{\alpha}^{-1})(f)(\alpha^{-1}(x),\alpha^{-1}(y))\\
  & =(u_{\epsilon_2}(\rho_{\lambda_2}^{\ast}\circ\widehat{\alpha}^{-1})(f)u_{\epsilon_2}^{-1})(\alpha^{-1}(x),\alpha^{-1}(y))\\
  & = u_{\epsilon_2\circ \alpha^{-1}}(x,x)\rho_{\lambda_2}^{\ast}(\widehat{\alpha}^{-1}(f))(\alpha^{-1}(x),\alpha^{-1}(y))
  u_{\epsilon_2\circ \alpha^{-1}}^{-1}(y,y)\\
  & = u_{\epsilon_2\circ \alpha^{-1}}(x,x)[\widehat{\alpha}^{-1}(f)(\lambda_2(\alpha^{-1}(y)),\lambda_2(\alpha^{-1}(x))]^\ast
  u_{\epsilon_2\circ \alpha^{-1}}^{-1}(y,y)\\
  & = u_{\epsilon_2\circ \alpha^{-1}}(x,x)[f(\lambda_1(y),\lambda_1(x))]^\ast u_{\epsilon_2\circ \alpha^{-1}}^{-1}(y,y)\\
  & = (\Psi_{u_{\epsilon_2\circ \alpha^{-1}}} \circ \rho_{\lambda_1}^{\ast})(f)(x,y)\\
  & = \rho_{\epsilon_2\circ \alpha^{-1}}(f)(x,y)
\end{align*}
and hence $\widehat{\alpha}\circ\rho_{\epsilon_2}\circ \widehat{\alpha}^{-1}=\rho_{\epsilon_2\circ \alpha^{-1}}$.
\end{proof}

\begin{theorem}\label{x_3morethan1}
Let $K$ be a field of characteristic different from $2$ and $X$ a connected poset such that every multiplicative automorphism of $FI(X,K)$ is inner. Let $\lambda_1,\lambda_2$ be involutions on $X$ such that $|X_3^{\lambda_1}|=|X_3^{\lambda_2}|> 1$. Let $\ast_i$ be an involution of the second on $K$, and $\epsilon_i:X_3^{\lambda_i} \to K_0^{\times}$ a map, $i=1,2$, such that the involutions $\rho_{\epsilon_1}=\Psi_{u_{\epsilon_1}} \circ \rho_{\lambda_1}^{\ast_1}$ and $\rho_{\epsilon_2}=\Psi_{u_{\epsilon_2}} \circ \rho_{\lambda_2}^{\ast_2}$ are not equivalent via inner automorphisms. Then $\rho_{\epsilon_1}$ and $\rho_{\epsilon_2}$ are equivalent if and only if $\ast_1=\ast_2$ and there exists an automorphism $\alpha$ of $X$ such that $\alpha\circ\lambda_2=\lambda_1\circ\alpha$ (i.e., $\lambda_1$ and $\lambda_2$ are equivalent) and $\rho_{\epsilon_1}$ and $\rho_{\epsilon_2\circ\alpha^{-1}}$ are equivalent, via inner automorphisms.
\end{theorem}
\begin{proof}
If $\rho_{\epsilon_1}$ and $\rho_{\epsilon_2}$ are equivalent, then by the proof of Corollary~\ref{cor_general} (i), there exists an automorphism $\alpha$ of $X$ such that $\rho_{\epsilon_1}$ and $\widehat{\alpha}\circ\rho_{\epsilon_2}\circ \widehat{\alpha}^{-1}$ are equivalent via inner automorphisms, $\alpha\circ\lambda_2=\lambda_1\circ\alpha$ and $\ast_1=\rho_{\epsilon_1}|_K=\rho_{\epsilon_2}|_K=\ast_2$. By Lemma~\ref{conjugado_rho_epsilon}, $\widehat{\alpha}\circ\rho_{\epsilon_2}\circ \widehat{\alpha}^{-1}=\rho_{\epsilon_2\circ \alpha^{-1}}$. Therefore, $\rho_{\epsilon_1}$ and $\rho_{\epsilon_2\circ\alpha^{-1}}$ are equivalent, via inner automorphisms.

Conversely, suppose that $\ast_1=\ast_2=:\ast$ and there exists an automorphism $\alpha$ of $X$ such that $\alpha\circ\lambda_2=\lambda_1\circ\alpha$, and $\rho_{\epsilon_1}$ and $\rho_{\epsilon_2\circ\alpha^{-1}}$ are equivalent, via inner automorphisms. By Lemma~\ref{conjugado_rho_epsilon}, $\rho_{\epsilon_2\circ \alpha^{-1}}=\widehat{\alpha}\circ\rho_{\epsilon_2}\circ \widehat{\alpha}^{-1}$. It follows that there is an inner automorphism $\Psi$ of $FI(X,K)$ such that $\Psi\circ\rho_{\epsilon_1}=\widehat{\alpha}\circ\rho_{\epsilon_2}\circ \widehat{\alpha}^{-1}\circ\Psi$, that is, $(\widehat{\alpha}^{-1}\circ\Psi)\circ\rho_{\epsilon_1}=\rho_{\epsilon_2}\circ (\widehat{\alpha}^{-1}\circ\Psi)$ and hence $\rho_{\epsilon_1}$ and $\rho_{\epsilon_2}$ are equivalent.
\end{proof}








\end{document}